\documentclass{aic}
\usepackage{amsmath}
\usepackage{newtxtext,newtxmath}
\theoremstyle{plain}
\newtheorem{theorem}{Theorem}[section]

\newtheorem{proposition}[theorem]{Proposition}

\newtheorem{conjecture}[theorem]{Conjecture}
\theoremstyle{definition}
\newtheorem{definition}[theorem]{Definition}

\newcommand{\ex}{{\rm{ex}}}

\newcommand{\Mod}[1]{\ (\mathrm{mod}\ #1)}
\newcommand{\vb}[1]{\boldsymbol{#1}}

\aicAUTHORdetails{%
  title = {Tight paths  in  convex geometric hypergraphs}, %% please capitalize all significant words
  author = {Zolt\'an F\"uredi, Tao Jiang, Alexandr Kostochka, Dhruv Mubayi and Jacques Verstra\"ete},
  plaintextauthor = {Zoltan Furedi, Tao Jiang, Alexandr Kostochka, Dhruv Mubayi and Jacques Verstraete},
  plaintexttitle = {Tight paths  in  convex geometric hypergraphs}, %%  title without math or LaTeX
  runningtitle = {Tight paths  in   hypergraphs},
  runningauthor = {Zolt\'an F\"uredi, Tao Jiang, Alexandr Kostochka, Dhruv Mubayi and Jacques Verstra\"ete},
  copyrightauthor = {Z. F\"uredi, T. Jiang, A. Kostochka, D. Mubayi, J. Verstra\"ete},
  keywords = {Paths, convex geometric hypergraphs, Turan number},
}   %%% END \aicAUTHORdetails

\aicEDITORdetails{%
   year={2020},
   number={1},
   received={24 June 2018},   % received date: example: 7 January 2017
   published={28 February 2020},  % published date
   doi={10.19086/aic.12044},      % XXX = number of paper, e.g. aic006 for paper#6
}   %%% END \aicEDITORdetails

\begin{document}

\begin{frontmatter}[classification=text]
%% EDITOR: this will force the keywords to appear right after the Abstract.
%%   If the abstract is too long and would force the keywords off the
%%   front page, please comment out % [classification=text] above
%%   This way the keywords will be floated on the bottom of the first page
%%   even though the Abstract spills over to the next page.

%%% AUTHOR: Title goes here.  This line is optional.  You must use it
%%   if title has footnote attached or requires nontrivial typesetting,
%%   e.g., inclusion of linebreaks to force nice layout.
\title{Tight paths  in  convex geometric hypergraphs}%\titlefootnote{This is a footnote to the title}} %% please capitalize all significant words

%%% AUTHOR:
%%% List all authors. If you wish, place grant acknowledgements in \thanks.
%%% In brackets include a short tag for each author.
\author[furedi]{ Zolt\'an F\" uredi\thanks{Supported by grant K116769
		from the National Research, Development and Innovation Office NKFIH and
		by the Simons Foundation Collaboration grant 317487.}}
\author[jiang]{	Tao Jiang\thanks{Supported by National Science Foundation award DMS-1400249.}}
\author[kostochka]{Alexandr Kostochka\thanks{Supported by NSF grant DMS1600592,
		by Arnold O. Beckman Award RB20003 of the University of Illinois at Urbana-Champaign and by grants 18-01-00353A and 19-01-00682 of the Russian Foundation for Basic Research.}}
	\author[mubayi]{Dhruv Mubayi\thanks{Supported by NSF grants  DMS-1300138 and DMS-1763317.}}
\author[verstraete]{Jacques Verstra\"ete\thanks{Supported by NSF grants DMS-1556524 and DMS-1800332.}
}

%%% AUTHOR: Abstract goes here
\begin{abstract}
	In this paper, we prove a theorem on tight paths in convex geometric hypergraphs, which is asymptotically sharp in infinitely many cases. Our geometric theorem is a common generalization of early results of Hopf and Pannwitz~\cite{Hopf-Pannwitz}, Sutherland~\cite{Sutherland}, Kupitz and Perles~\cite{Kupitz-Perles} for convex geometric graphs, as well as the classical Erd\H{o}s-Gallai Theorem~\cite{Erdos-Gallai} for graphs. As a consequence, we obtain the first substantial improvement on the Tur\'{a}n problem for tight paths in uniform hypergraphs.
\end{abstract}
\end{frontmatter}

\section{Introduction}

In this paper, we address extremal questions for tight paths in uniform hypergraphs and in convex geometric hypergraphs.
For $k \geq 1$ and $r \geq 2$, a {\em tight $k$-path} is an $r$-uniform hypergraph (or simply $r$-graph) $P_k^r = \{v_iv_{i+1}\dots v_{i+r-1} : 0 \leq i < k\}$. Let $\ex(n,P_k^r)$ denote the maximum number of edges in an $n$-vertex
$r$-graph not containing a tight $k$-path. It appears to be difficult to determine $\ex(n,P_k^r)$ in general, and even the asymptotics as $n \rightarrow \infty$ are not known.
The following is a special case of a conjecture of Kalai~\cite{FF} on tight trees, generalizing the well-known Erd\H{o}s-S\'{o}s Conjecture~\cite{Erdos-Sos-Tree}:

\begin{conjecture} [Kalai] \label{kalaiconj}
	For $n \geq r \geq 2$ and $k \geq 1$, $\ex(n,P_k^r) \leq \frac{k-1}{r}{n \choose r - 1}$.
\end{conjecture}

A construction based on combinatorial designs shows this conjecture if true is tight -- the existence of designs was established by Keevash~\cite{Keevash} and also more recently by Glock, K\"{u}hn, Lo and Osthus~\cite{GKLO}.
It is straightforward to see that any $n$-vertex $r$-graph $H$ that does not contain a tight $k$-path has at most $(k - 1){n \choose r - 1}$ edges. Patk\'{o}s~\cite{Patkos} gave an improvement over this
bound in the case $k < 3r/4$. In the special case $k = 4$ and $r = 3$, it is shown in~\cite{FJKMV} that $\ex(n,P_4^3) = {n \choose 2}$ for all $n \geq 5$. In this paper, we give the first non-trivial upper bound on $\ex(n,P_k^r)$ valid for all $k$ and $r$:

\begin{theorem}\label{tight-path}
	For $n \geq 1$, $r \geq 2$, and $k \geq 1$,
	\[ \ex(n,P_k^r)  \leq \left\{\begin{array}{ll}
	\frac{k-1}{2}{n \choose r - 1} & \mbox{ if }r\mbox{ is even} \\
	\frac{1}{2}(k + \lfloor \frac{k-1}{r}\rfloor){n \choose r - 1} & \mbox{ if }r\mbox{ is odd}
	\end{array}\right.\]
\end{theorem}

The case $r = 2$ of this result is the well-known Erd\H{o}s-Gallai Theorem~\cite{Erdos-Gallai} on paths in graphs. We prove Theorem \ref{tight-path}
by introducing a novel method for extremal problems for paths in convex geometric hypergraphs.

\medskip

{\bf Convex geometric hypergraphs.}  A {\em convex geometric hypergraph} (or cgh for short) is an $r$-graph whose vertex set is a set $\vb{\Omega}_n$ of $n$ vertices in strictly convex position in the plane, and whose edges are viewed as convex $r$-gons with vertices from $\vb{\Omega}_n$. Given an $r$-uniform cgh $F$, let $\ex_{\circlearrowright}(n,F)$ denote the maximum number of edges in an $n$-vertex $r$-uniform cgh that does not contain $F$.
Extremal problems for convex geometric graphs (or cggs for short) have been studied extensively, going back to theorems in the 1930's on disjoint line segments in the plane.
We refer the reader to the papers of Bra\ss, K\'{a}rolyi and Valtr~\cite{Brass-Karolyi-Valtr}, Capoyleas and Pach~\cite{Capoyleas-Pach} and the references therein for many related extremal problems on convex geometric graphs and to Aronov, Dujmovi\v{c}, Morin, Ooms and da Silveira~\cite{Aronov}, Bra\ss~\cite{Brass}, Brass, Rote and Swanepoel~\cite{Brass-Rote-Swanepoel}, and Pach and Pinchasi~\cite{Pach-Pinchasi} for problems in convex geometric hypergraphs, and their connections to important problems in
discrete geometry, as well as the triangle-removal problem (see Aronov, Dujmovi\v{c}, Morin, Ooms and da Silveira~\cite{Aronov} and Gowers and Long~\cite{Gowers-Long}).

\medskip

Concerning results on convex geometric graphs, let $\mathsf M_k$ denote the cgg consisting of $k$ pairwise disjoint line segments. Generalizing results of Hopf and Pannwitz~\cite{Hopf-Pannwitz} and Sutherland~\cite{Sutherland}, Kupitz~\cite{Kupitz} and Kupitz and Perles~\cite{Kupitz-Perles}
showed that for $n \geq k\ge 2$,
\[ \ex_{\circlearrowright}(n,\mathsf M_k) \leq (k - 1)n.\]
Perles proved the following even stronger theorem. Define a {\em $k$-zigzag} $\mathsf P_k$ to be a $k$-path $v_0 v_1 \dots v_k$ with vertices in $\vb{\Omega}_n$
such that in a fixed cyclic ordering of $\vb{\Omega}_n$, the vertices appear in the order $v_0, v_2, v_4, \dots, v_5, v_3, v_1$, $v_0$ (see the left picture in Figure 1).

\begin{theorem}[Perles] \label{perlesthm}  For $n, k \ge 1$,  $\ex_{\circlearrowright}(n,\mathsf P_k) \leq (k - 1)n/2$.
\end{theorem}

The bound in Theorem~\ref{perlesthm}  is tight when $k$ divides $n$ since any disjoint union of cliques of order $k$ does not contain any path with $k$ edges.
In particular, since $\mathsf P_{2k-1}$ contains $\mathsf M_k$, Theorem~\ref{perlesthm} implies $\ex_{\circlearrowright}(n,\mathsf M_k) \leq \ex_{\circlearrowright}(n,\mathsf P_{2k-1}) \leq (k - 1)n$. It appears to be challenging to determine for all $k$ and $r$ the exact value of the extremal function or the extremal cghs without $k$-zigzag (see Keller and Perles~\cite{Chaya-Perles} for a discussion of extremal constructions in the case $r = 2$).

\medskip

In this paper, we generalize Theorem \ref{perlesthm}
to convex geometric hypergraphs, and use its proof technique to prove Theorem \ref{tight-path}. We let $\prec$ denote a fixed cyclic ordering of the vertices of $\vb{\Omega}_n$, and let $[u,v] = \{w \in \vb{\Omega}_n : u \prec w \prec v\}$ denote a {\em segment} of $\vb{\Omega}_n$. If $I_1, I_2, \ldots \subset \vb{\Omega}_n$, then we write $I_1 \prec I_2 \prec \cdots$ if all vertices of $I_j$ are followed in the  ordering $\prec$ by all vertices of $I_{j+1}$ for $j \geq 1$.  We use the following definition of a path in a convex geometric hypergraph:

\begin{definition} [Zigzag paths] \label{defz}
	For $k \geq 1$ and even $r \geq 2$, a tight $k$-path $v_0 v_1 \dots v_{k + r - 2}$ with vertices in $\vb{\Omega}_n$ is a {\em $k$-zigzag}, denoted $\mathsf P_k^r$, if
	there exist disjoint segments $I_0 \prec I_1 \prec \dots \prec I_{r-1}$ of $\vb{\Omega}_n$ such that $\{v_i : i \equiv j \Mod{r}\} \subseteq I_j$ for $0 \leq j < r$ and
	\vspace{-0.1in}
	\begin{center}
		\begin{tabular}{lp{5.8in}}
			{\rm (i)}  & if $j$ is even, then $v_j  \prec  v_{j+r}  \prec  v_{j + 2r} \prec  \cdots$. \\
			{\rm (ii)} & if $j$ is odd, then $v_j \succ v_{j+r} \succ  v_{j+2r} \succ \cdots$.
		\end{tabular}
	\end{center}
\end{definition}
\vspace{-0.1in}

In words, the vertices of the zigzag with subscripts congruent to $j \Mod{r}$ appear in increasing order of subscripts if $j$ is even, followed by the vertices with subscripts congruent to $j + 1 \Mod{r}$ in decreasing order of subscripts with respect to the cyclic ordering $\prec$.
In the case of graphs, a $k$-zigzag is simply $\mathsf P_k^2 = \mathsf P_k$ from Theorem \ref{perlesthm}. We give examples of zigzag paths $\mathsf P_6^2$ and $\mathsf P_5^4$ in Figure \ref{fig:zigzags} below (the last edge of each path is indicated in bold).

\begin{figure}[!ht]
	\begin{center}
		\includegraphics[width=4in]{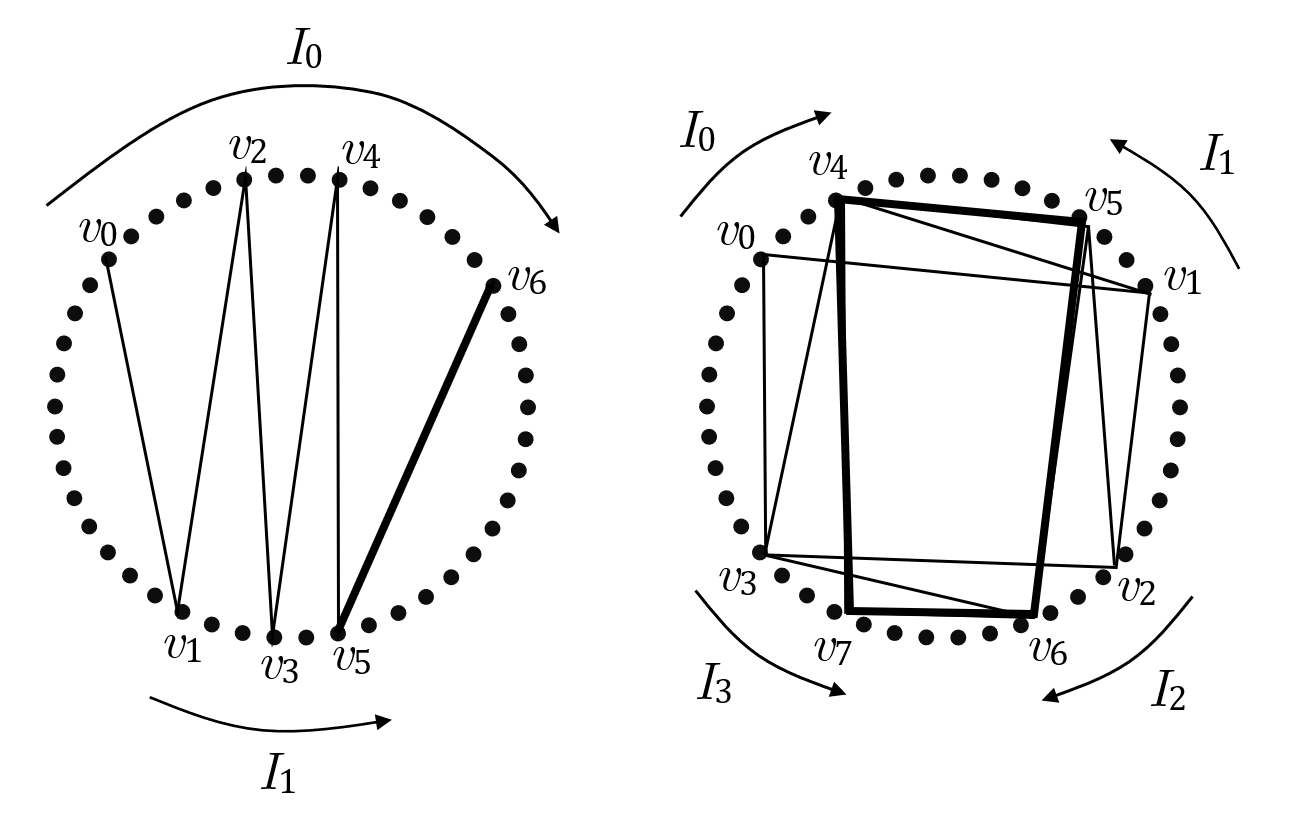}
		\caption{Zigzag paths}
		\label{fig:zigzags}
		
	\end{center}
	
\end{figure}

\vspace{-0.2in}

The following result generalizes Theorem~\ref{perlesthm} to $r$-uniform cghs when $r$ is even:

\begin{theorem}  \label{mainzigzag}
	Let $n, k \ge 1$, and let $r \geq 2$ be even. Then
	\[ \ex_{\circlearrowright}(n,\mathsf P_k^r) \leq \frac{(k-1)(r - 1)}{r}{n \choose r-1}.\]
\end{theorem}

This theorem is asymptotically sharp in infinitely many cases, and is a common generalization Theorem \ref{perlesthm} and the Erd\H{o}s-Gallai Theorem~\cite{Erdos-Gallai}. The proof of  Theorem \ref{mainzigzag} is also the basis for our proof of Theorem \ref{tight-path}.

\medskip

{\bf Organization.} This paper is organized as follows. In Section \ref{extend}, we give a method for extending a $k$-zigzag in an $r$-uniform cgh to a $(k + 1)$-zigzag. This is used in the short proof of Theorem \ref{mainzigzag} in Section \ref{even}. In Section \ref{lb}, we give constructions of dense
cghs without $k$-zigzags, and in Section~\ref{proof-tightpath}, we prove Theorem \ref{tight-path} using the proof technique of Theorem \ref{mainzigzag}.

\medskip

{\bf Notation.} We let $\vb{\Omega}_n$ denote a generic set of $n$ points in strictly convex position in the plane, and let $\prec$ denote a cyclic ordering of $\vb{\Omega}_n$. For $u,v \in \vb{\Omega}_n$, we write $[u,v] = \{w : u \prec w \prec v\}$; this is the set of vertices in the segment of $\vb{\Omega}_n$ from $u$ to $v$ (including $u$ and $v$) in the ordering $\prec$. For $u,v \in \vb{\Omega}_n$, let
$\ell(u,v) = \min\{|[u,v]| - 1,|[v,u]| - 1\}$. In other words,  $\ell(u,v)$ is the number of sides in a shortest segment of $\vb{\Omega}_n$ between $u$ and $v$.
Throughout this paper, cghs have vertex set in $\vb{\Omega}_n$ with cyclic ordering $\prec$.  For an $r$-uniform cgh $F$, let $\ex_{\circlearrowright}(n,F)$ denote the maximum number of edges in an $r$-uniform cgh on $\vb{\Omega}_n$ that does not contain an ordered substructure isomorphic to $F$. We write $V(H)$ for the vertex set of a hypergraph $H$,
and represent the edges as unordered lists of vertices. We identify a hypergraph $H$ with its edge-set, denoting by $|H|$ the number of edges in $H$.
For $v \in V(H)$, the {\em neighborhood of $v$} is $N(v) = \bigcup_{v \in e \in H} (e \backslash \{v\})$.
Let $\partial H$ denote the {\em shadow} of an $r$-graph $H$, namely $\{e \backslash \{x\} : x \in e \in H\}$.

\section{Extending zigzags}\label{extend}

\subsection{Extending zigzags in graphs}

We start with a short proof of Theorem \ref{perlesthm} for zigzags of odd length, along the lines of Perles' proof, which gives an idea of the proof of Theorem \ref{mainzigzag}.

\begin{proposition}
	Let $k \geq 0$. If $G$ is an $n$-vertex cgg with no $(2k+1)$-zigzag, then $|G|  \leq kn$.
\end{proposition}
\vspace{-0.15in}
\begin{proof} Proceed by induction on $k$; for $k = 0$, the statement is clear. Suppose $k \geq 1$ and $G$ is an $n$-vertex cgg with no $(2k + 1)$-zigzag. For $v \in V(G)$, let
	$f(v)$ be the first vertex of $N(v)$ after $v$ in the ordering $\prec$. Let $E = \{vf(v) : v \in V(G)\}$. If $v_0 v_1 \dots v_{2k-1}$ is a $(2k - 1)$-zigzag in $F = G \backslash E$, then $f(v_0) v_0 \dots v_{2k - 1} f(v_{2k - 1})$ is a $(2k + 1)$-zigzag in $G$. So $F$ has no $(2k - 1)$-zigzag, and $|F| \leq (k - 1)n$ by induction. Since $|E| \leq n$, $|G| = |F| + |E| \leq kn$.
\end{proof}

A key point is that a zigzag $v_0 v_1 \dots v_k$ can be extended to a $(k + 1)$-zigzag $v_0 v_1 \dots v_k v$ if $v$ is adjacent to $v_k$ and $v \in [v_k,v_{k-1}]$ if $k$ is even, whereas $v \in [v_{k-1},v_k]$ if $k$ is odd (the reader may find it helpful to
refer to Figure 1). In the next section, we generalize these ideas to uniform cghs.

\subsection{Extending zigzags in hypergraphs}

Fixing an even $r \geq 2$, we write $\vb{v}_k$ as shorthand for $(v_{k-1},v_k,\dots,v_{k+ r - 2})$. We use this as notation for the ordering of the last edge of
a $k$-zigzag:

\begin{definition}
	The {\em end} of a $k$-zigzag $v_0 v_1 \dots v_{k + r- 2}$ is $\vb{v}_k = (v_{k-1},v_k,\dots,v_{k+r-2})$. Let $I(\vb{v}_k) = [v_{k - 1},v_k]$ if $k$ is odd and $I(\vb{v}_k) = [v_{k + r - 2},v_{k - 1}]$ if $k$ is even, and
	\[X(\vb{v}_k) = \{v \in I(\vb{v}_k) : vv_kv_{k+1}\dots v_{k+r-2} \in H\}\]
\end{definition}
Referring to Figure 1, in the picture on the left $X(\vb{v}_6)$ is the set of $v$ in the segment from $v_6$ to $v_5$ clockwise such that $v_6v$ is an edge. In the picture
on the right, $X(\vb{v}_5)$ is the set of $v$ in the segment from $v_4$ to $v_5$ clockwise such that $v_5 v_6 v_7 v$ is an edge.
In the next proposition, we see that any vertex in $X(\vb{v}_k)$ can be used to ``extend'' a $k$-zigzag ending in $\vb{v}_k$ to a $(k + 1)$-zigzag:

\begin{proposition} \label{obs}
	Let $\vb{v}_k \in V(H)^r$ be the end of a $k$-zigzag $\mathsf P$ in $H$. Then for any $v_{k + r - 1} \in X(\vb{v})$,
	$\mathsf P \cup \{v_kv_{k+1}\dots v_{k+r-1}\}$ is a $(k + 1)$-zigzag ending in $\vb{v}_{k+1}$.
\end{proposition}

\begin{proof} Let $\mathsf P = v_0 v_1 \dots v_{k + r - 2}$ and let $I_0 \prec I_1 \prec \dots \prec I_{r - 1}$ be the segments in Definition \ref{defz}.
	Let $v_{k - 1} \in I_j$, so $j \equiv k - 1 \Mod{r}$. If $k$ is odd, then $j$ is even, and the vertices of $I_j \cup I_{j+1}$ appear in the order $v_j \prec \cdots  \prec v_{k-1} \prec  v_k \prec \cdots \prec v_{j+1}$ by Definition \ref{defz}(i). Then for any $v_{k + r - 1} \in X(\vb{v}_k)$, $e = v_kv_{k + 1}\dots  v_{k + r - 2}v_{k+r-1} \in H$ and adding $e$ to $\mathsf P$ and $v_{k + r - 1}$ to $I_j$ before $v_{k - 1}$ in the clockwise orientation, we obtain a $(k + 1)$-zigzag. Similarly, if $k$ is even, then $j$ is odd so the vertices of $I_{j-1} \cup I_{j}$ appear in the order $v_{j-1} \prec \cdots \prec v_{k+r-2} \prec v_{k-1} \prec \cdots \prec v_{j}$ by Definition \ref{defz}(ii), and we add $e$ to $\mathsf P$ and $v_{k + r - 1}$ after $v_{k - 1}$ in $I_j$ in the clockwise orientation.
\end{proof}

\medskip

\begin{definition} Let $S_k(H)$ be the set of ends $\vb{v}_k \in V(H)^r$ of $k$-zigzags in $H$, and
	\[
	T_k(H) = \{\vb{v}_k \in S_k(H) : X(\vb{v}_k) = \emptyset\}.
	\]
\end{definition}

Informally, $T_k(H)$ is the set of ends of $k$-zigzags which cannot be ``extended'' to $(k + 1)$-zigzags. The two key propositions for the proof of Theorem \ref{mainzigzag}
are as follows.

\medskip

\begin{proposition}\label{injection}
	For $\vb{v}_k \in S_k(H)\setminus T_k(H)$, let $v_{k+r-1} \in X(\vb{v}_k)$ be as close as possible to $v_{k-1}$ in the segment $I(\vb{v}_k)$. Then $f(\vb{v}_k) = \vb{v}_{k+1}$ is an injection from $S_k(H) \backslash T_k(H)$ to $S_{k+1}(H)$. In particular,
	\begin{equation}\label{skbound1}
	|S_{k + 1}(H)| \geq |S_k(H) \backslash T_k(H)|.
	\end{equation}
\end{proposition}

\begin{proof}
	By Proposition \ref{obs}, $f(\vb{v}_k) \in S_{k + 1}(H)$. Furthermore, $f(\vb{v}_k) = f(\vb{w}_k)$ implies
	$\vb{v}_{k + 1} = \vb{w}_{k + 1}$, which gives $v_i = w_i$ for $k \leq i \leq k + r - 1$. If $v_{k - 1} \neq w_{k - 1}$,
	then either $w_{k - 1}$ is closer to $v_{k - 1}$ than $w_{k + r - 1}$ in $I(\vb{v}_k)$, or $v_{k - 1}$ is closer to $w_{k - 1}$ than $v_{k + r - 1}$
	in $I(\vb{v}_k)$. These contradictions imply $v_{k - 1} = w_{k - 1}$, and so $\vb{v}_k = \vb{w}_k$ and $f$ is an injection.
\end{proof}

\smallskip

\begin{proposition}\label{injection2}
	For $\vb{v}_k \in T_k(H)$, the map $g(\vb{v}_k) = (v_k,v_{k + 1},\dots,v_{k + r - 2})$ is an injection from $T_k(H)$ to cyclically ordered elements of $\partial H$. In particular,
	\begin{equation}\label{tkbound1}
	|T_k(H)| \leq (r - 1)|\partial H|.
	\end{equation}
\end{proposition}

\begin{proof}
	If $g(\vb{v}_k) = g(\vb{w}_k)$, then $w_i = v_i$ for $k \leq i \leq k + r - 2$. Suppose $v_{k - 1} \neq w_{k - 1}$. Then either
	$v_{k+r-2} \prec w_{k-1} \prec v_{k-1}$, and $v_{k-1} \in X(\vb{w}_k)$, or $v_{k-1} \prec w_{k-1} \prec v_k$, and $w_{k-1} \in X(\vb{v}_k)$. In either case,
	$\vb{v}_k \not \in T_k$ or $\vb{w}_k \not \in T_k$, a contradiction. So $v_{k - 1} = w_{k - 1}$, which implies $\vb{v}_k = \vb{w}_k$.
\end{proof}

\bigskip

\section{Proof of Theorem \ref{mainzigzag} on zigzags}\label{even}

The following theorem implies Theorem \ref{mainzigzag}, since if $H$ is an $n$-vertex $r$-uniform cgh not containing
a $k$-zigzag, then $S_k(H) = \emptyset$, and we always have $|\partial H| \leq {n \choose r - 1}$.

\begin{theorem}\label{zigzag}
	Let $k \geq 1$ and let $r \geq 2$ be even. Then for any $r$-uniform cgh $H$,
	\begin{equation}\label{skbound}
	|S_k(H)| \geq r|H| - (r - 1)(k - 1)|\partial H|.
	\end{equation}
\end{theorem}

\begin{proof} We prove (\ref{skbound}) by induction on $k$. Let $k = 1$ and $e \in H$. By Definition \ref{defz}(i), there are $r$ possible orderings
	of the vertices of $e$ giving a 1-zigzag:  having chosen the first vertex, the ordering of the remaining vertices of $e$ is determined. Therefore $|S_1(H)| \geq r|H|$. For the induction step, suppose $k \geq 1$ and (\ref{skbound}) holds.
	By (\ref{skbound1}) and (\ref{tkbound1}),
	\begin{eqnarray*}
		|S_{k+1}(H)| &\geq& |S_k(H) \backslash T_k(H)| \geq r|H| - (r - 1)(k - 1)|\partial H| - |T_k(H)| \\
		&\geq& r|H| - (r - 1)k|\partial H|.
	\end{eqnarray*}
	This proves (\ref{skbound}).
\end{proof}

\bigskip

\section{Stack-free constructions}\label{lb}

Let $k \geq 1$ and let $r \geq 2$ be even. A {\em $k$-stack}, denoted $\mathsf M_k^r$, consists of edges $\{v_{ir},v_{ir+1},\dots,v_{ir+r-1} : 0 \leq i < k\}$ where $v_0 v_1 \dots v_{(k - 1)r-1}$ is an $r$-uniform zigzag path; in other words we pick every $r$th edge from a zigzag path $\mathsf P_{(k-1)r + 1}^r$.
An example for $r = 4$ and $k = 7$ is shown below, where the extreme points on the perimeter form $\vb{\Omega}_{28}$.

\begin{figure}[!ht]
	\begin{center}
		\includegraphics[width=2.5in]{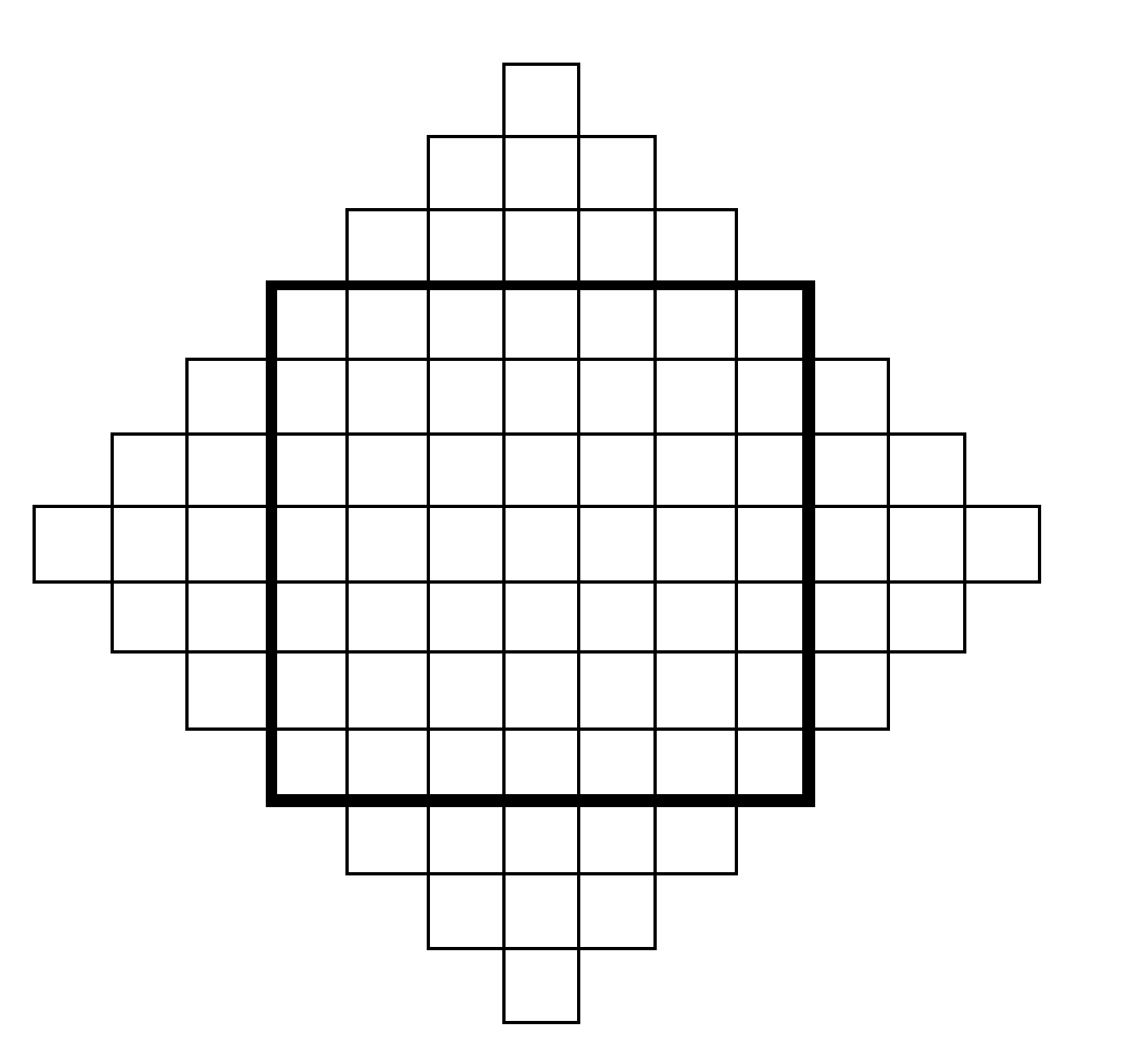}
		\begin{center}
			$\mathsf M_7^4$
		\end{center}
		
		\caption{Stack}
		\label{fig:stacks}
		
	\end{center}
\end{figure}

There is a simple construction of an $r$-uniform cgh with no $k$-stack when $k$ is odd with $(k - 1)(r - 1){n \choose r - 1} + O(n^{r - 2})$ edges.
If $k \geq 3$ is odd, let $H$ be the cgh consisting of $r$-sets $e$ from $\vb{\Omega}_n$ such that $\ell(u,v) \leq k - 1$ for some $u,v \in e$. It is straightforward to see that $|H| = (r - 1)(k - 1){n \choose r - 1} + O(n^{r - 2})$, and
$H$ contains no $k$-stack since the ``middle'' edge $e$ in the stack -- drawn in bold in Figure 2 -- has $\ell(u,v) \geq k$ for all $u,v \in e$.

\medskip

In this section, we extend this construction to all values of $k$, thereby proving the following theorem, which may be of independent interest. In particular, this construction does not contain $\mathsf P_{(k - 1)r + 1}^r$, and shows Theorem \ref{mainzigzag} is asymptotically tight for zigzags of length $1 \Mod{r}$.

\medskip

\begin{theorem}\label{stack}
	Let $k \geq 1$ and $r \geq 2$ be even. Then
	\[ \ex_{\circlearrowright}(n,\mathsf M_k^r) = (k - 1)(r - 1){n \choose r-1} + O(n^{r - 2}).\]
\end{theorem}

\begin{proof} We have $\ex_{\circlearrowright}(n,\mathsf M_k^r) \leq (k -1)(r - 1){n \choose r - 1}$ from Theorem \ref{mainzigzag}. The main part of the proof is the construction of an $r$-uniform cgh with
	$(k - 1)(r - 1){n \choose r - 1} + O(n^{r - 2})$ edges that does not contain a $k$-stack. It will be convenient to let $\vb{\Omega}_n = \{0,1,2,\dots,n-1\}$ in cyclic order, and view our edges as ordered $r$-tuples $(v_{0},v_{1},\ldots,v_{r-1})$ where
	$0\leq v_0 < v_1 < \ldots < v_{r-1} \leq n-1$.
	
	\medskip
	
	Our construction $H = H(n,r,k)$ has the form $H=\bigcup_{j=0}^{k-1} H_j$, where
	
	\begin{tabular}{cp{5.5in}}
		(i) & $H_0=\{(v_0,v_{1},v_{2},\ldots,v_{r-1})\, : \; v_0 = 0\}$,\\
		(ii) & $H_j = \bigcup_{h = 0}^{r-1} \{(v_0,v_1,\ldots,v_{r-1}) \not \in H_0 \, : \, \ell(v_h,v_{h+1}) = j\}$ for $1 \leq j \leq k - 2$, \\
		(iii) & $H_{k-1}= \bigcup_{h = 1}^{r/2 - 1} \{(v_{0},v_{1},\ldots,v_{r-1}) \not \in H_0 \, : \; \ell(v_{2h - 1},v_{2h}) \in \{k-1,k\} \}$.
	\end{tabular}
	
	\medskip
	
	{\bf Claim 1.}\label{size} $|H| = (k-1)(r-1){n\choose r-1} + O(n^{r - 2})$.
	
	\smallskip
	
	{\bf Proof.} By definition, $|H_0|={n-1\choose r-1}$, and $H_0 \cap \bigcup_{j=1}^{k-1} H_j = \emptyset$. For any $j : 1\leq j\leq k-2$, as $n \rightarrow \infty$,
	\[ |H_j| = (n-1){n-j-1\choose r-2} + O(n^{r - 2}) =  (r-1){n\choose r-1} + O(n^{r - 2})\]
	and also
	\[ |H_{k-1}| = 2(r/2 - 1){n-1\choose r-1} + O(n^{r - 2}) = (r-2){n\choose r-1} + O(n^{r - 2}).\]
	
	If $1 \leq i < j \leq k - 1$, $|H_i\cap H_j|=O(n^{r-2})$. By inclusion-exclusion,
	\[ |H| \geq |H_0| + \sum_{j = 1}^{k - 1} |H_j| - \sum_{i < j} |H_i \cap H_j| = (k - 1)(r - 1){n \choose r - 1} + O(n^{r - 2}).\]
	This proves the claim.\qed
	
	\medskip
	
	{\bf Claim 2.} {\em $\mathsf M_k^r \not \subseteq H$.}
	
	\smallskip
	
	{\bf Proof.} Suppose $H$ contains a $k$-stack. The key is to consider the ``middle'' two edges of the stack, say
	$e$ and $f$. Then the vertex 0 is in at most one of $e$ and $f$. If $v_0$ is the first vertex of $e$ and $w_0$ is the first vertex
	of $f$ after 0 in the clockwise direction, then without loss of generality we may assume $v_0 < w_0$. Now consider the pairs
	$w_1 w_2$, $w_3 w_4$ up to $w_{r-1}w_{r}$ which are in $f$. We claim all of these pairs have length at least $k + 1$, contradicting the definition of $H$,
	since $f$ would then not be a member of $H$. To see the claim, fix $h : 1 \leq h < r/2$. Notice that there are $k/2$ edges of the stack (excluding $f$) which contain a pair
	of vertices in the segment $[w_{2h-1},w_{2h}]$, and these pairs are vertex disjoint.
	However, then $\ell(w_{2h-1},w_{2h}) \geq 2(k/2 + 1) - 1 = k + 1$, and this holds for $1 \leq h < r/2$.
\end{proof}

%\medskip
%{\color{red} HAVENT CHECKED THIS}
%{\bf Construction 3: $\boldsymbol{r = 2 \Mod{4}}$.} We take Construction 2; the only difference when $r = 2 \Mod{4}$ is that $|C| \sim \frac{r-1}{r} \lfloor \frac{r}{4} \rfloor {n \choose r - 1}$, and this leads to
%$$\ex_{\c}(n, ZM_k^r) \geq (1 - o(1))(r-1)(k-1 - \tfrac{1}{2(r - 1)}){n \choose r-1}.$$

\section{Proof of Theorem \ref{tight-path} on tight paths}\label{proof-tightpath}

{\bf Proof for $r$ even.} Let $H$ be an $n$-vertex $r$-graph with no tight $k$-path, where $r$ is even. We aim to prove the following, which
gives Theorem \ref{tight-path} for $r$ even:

\begin{equation}\label{hbound}
|H| \leq \frac{k - 1}{2}|\partial H|.
\end{equation}

We follow the approach used to prove  Theorem \ref{mainzigzag} on a carefully chosen subgraph $G$ of $H$. This subgraph is defined via a random partition of $V(H)$:
let $s = r/2$ and let $\chi : V(G) \rightarrow \{0,1,\dots,s-1\}$ be a random $s$-coloring of the vertices of $H$ such that $\mathrm{P}(\chi(v) = i) = 1/s$ for $0 \leq i \leq s - 1$ and each vertex $v \in V(H)$,
and such that vertices are colored independently. Let $B_i = \{v \in V(H) : \chi(v) = i\}$, and define the following (random) subgraph of $H$:
\[ G = \{e \in H : |e \cap B_i| = 2 \mbox{ for }0 \leq i \leq s - 1\}.\]
In other words, each edge of $G$ has two vertices in each of the sets $B_i$.
For $0 \leq i \leq s - 1$, let
$$\partial_iG = \{e \in \partial G : |e \cap B_i| = 1\} \subset
\{e \in \partial H : |e \cap B_i| = 1, |e \cap B_j|=2 \hbox{ for } j \ne i\}
.$$ Then we have the following expected values:
\begin{equation}\label{tb}
\mathrm{E}(|G|) = \frac{r!}{2^s s^r} |H| \quad \mbox{ and } \quad \mathrm{E}(|\partial_i G|) \le  \frac{(r - 1)!}{2^{s-1}s^{r-1}}|\partial H|.
\end{equation}
The next step is to introduce some geometric structure on $G$. Let $\prec$ denote a cyclic ordering of the vertices of each of $B_0, B_1, \dots, B_{s - 1}$.

\begin{definition}[Good paths]\label{gp}
	We call a tight path  $v_0v_1\ldots v_{k+r-2}$ in $G$ {\em good} if
	\vspace{-0.1in}
	\begin{center}
		\begin{tabular}{lp{5.5in}}
			{\rm (i)} & for $0 \leq j < k + r - 2$, $v_j,v_{j + 1} \in B_i$ whenever $j \equiv 2i \Mod{r}$.\\
			{\rm (ii)} & the cyclic order in $B_i$ is always $v_j \prec v_{j+r} \prec v_{j+2r} \prec \ldots \prec v_{j+1+2r} \prec v_{j+1+r} \prec v_{j+1}$.
		\end{tabular}
	\end{center}
\end{definition}

An $r$-uniform good path with $k$ edges is shown in Figure 3, for $r = 6$ and $k = 4$.

\begin{figure}[!ht]
	\begin{center}
		\includegraphics[width=2.8in]{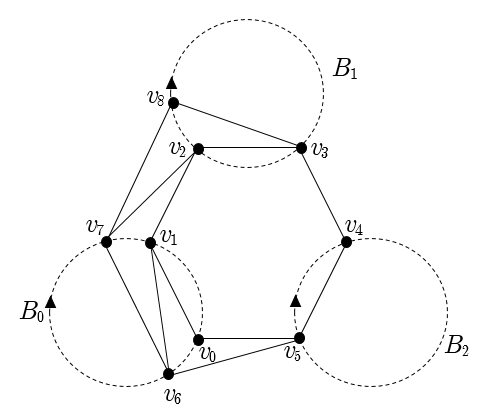}
		\caption{Good paths}
		\label{fig:good}
	\end{center}
\end{figure}

We now follow the ideas in Section \ref{extend}. By Definition \ref{gp}(i), $v_j \in B_i$ if and only if $i = h(j) = \lfloor j/2\rfloor \Mod{s}$.
Let $i = h(k - 1)$, so that $v_{k - 1} \in B_i$. We write $[u,v] = \{w \in B_i : u \prec w \prec v\}$.
Define $I(\vb{v}_k) = [v_{k - 1},v_k] \subseteq B_i$ if $k$ is odd and $I(\vb{v}_k) = [v_{k + r - 2},v_{k - 1}] \subseteq B_i$ if $k$ is even, and
\[
X(\vb{v}_k) = \{v \in I(\vb{v}_k) : vv_kv_{k+1}\dots v_{k+r-2} \in H\}
\]
Note that the definition of $X(\vb{v_k})$ is identical to that in Section \ref{extend} but with respect to the ordering $\prec$ of $B_i$, where $i = h(k - 1)$, and in particular, $I(\vb{v}_k), X(\vb{v}_k) \subseteq B_i$.
In Figure 3, $X(\vb{v}_4)$ consists of all verteices $v \in B_1$ clockwise from $v_8$ to $v_3$ such that $v_4 v_5 v_6 v_7 v_8 v \in G$.
Let $S_k(G)$ be the set of the ends of good $k$-paths in $G$, and let $T_k(G) = \{\vb{v}_k \in S_k(H) : X(\vb{v}_k) = \emptyset\}$.

\medskip

{\bf Claim 1.} {\em For $k \geq 1$, if $i = h(k - 1)$, then}
\begin{equation}\label{tk2}
|T_k(G)| \leq 2^{s-1}|\partial_i G|.
\end{equation}

{\bf Proof.} If $\vb{v}_k \in S_k(G)$, then $v_{k - 1} \in B_i$ since $i = h(k - 1)$. For $\vb{v}_k \in T_k(G)$, define
$g(\vb{v}_k) = (v_k,v_{k + 1},\dots$, $v_{k + r - 2})$. Then $v_k v_{k + 1} \dots v_{k + r - 2} \in \partial_iG$ and
$(v_k,v_{k + 1},\dots,v_{k + r - 2})$ is uniquely determined by specifying the order of the pairs $\{v_k,v_{k + 1},\dots,v_{k + r - 2}\} \cap B_j$
for each $j \neq i$. Therefore $g(\vb{v}_k)$ injectively maps elements of $T_k(G)$ to ordered elements of $\partial_i G$, where each element of $\partial_iG$ is
ordered in $2^{s-1}$ ways. We conclude $|T_k(G)| \leq 2^{s-1}|\partial_i G|$. \qed

\medskip

{\bf Claim 2.} {\em For $k \geq 1$, }
\begin{equation}\label{induction}
|S_k(G)| \geq 2^s |G| - 2^{s-1} \sum_{i = 0}^{k-2} |\partial_{h(i)}G|.
\end{equation}

{\bf Proof.} For $k = 1$, we observe for $e \in G$, there are two ways to label the pair $e \cap B_i$ for each $i \in [s]$, and therefore $|S_1(G)| \geq 2^s |G|$.
Suppose (\ref{induction}) holds for some $k \geq 1$. Then we copy the proofs of Propositions \ref{obs} and \ref{injection} to obtain
$|S_{k + 1}(G)| \geq |S_k(G) \backslash T_k(G)|$. By the induction hypothesis (\ref{induction}) and Claim 1,
\begin{eqnarray*}
	|S_{k+1}(G)| \; \; \geq \; \;  |S_k(G) \backslash T_k(G)| &\geq& 2^s |G| - 2^{s-1} \sum_{i = 0}^{k-2} |\partial_{h(i)}G| - |T_k(G)| \\
	&\geq& 2^s |G| - 2^{s-1} \sum_{i = 0}^{k - 1} |\partial_{h(i)}G|.
\end{eqnarray*}
This completes the induction step and proves (\ref{induction}). \qed

\medskip

{\bf Proof of (\ref{hbound}).} Finally we prove (\ref{hbound}). Taking expectations on both sides of (\ref{induction}), and using (\ref{tb}) and the linearity of expectation:
\begin{equation}\label{expect}
\mathrm{E}(|S_k(G)|) \geq 2^s \mathrm{E}(|G|) - 2^{s-1} \sum_{i = 0}^{k - 2} \mathrm{E}(|\partial_{h(i)}G|) \geq \frac{r!}{s^r} |H| -  \frac{(r - 1)!(k - 1)}{s^{r-1}} |\partial H|.
\end{equation}
Since $G \subseteq H$ has no tight $k$-path, $S_k(G) = \emptyset$. Using this in (\ref{expect}), we obtain (\ref{hbound}). \qed

\medskip

{\bf Proof for $r$ odd.} Let $H$ be an $n$-vertex $r$-graph containing no tight $k$-path. We aim to show
\begin{equation}\label{hbound2}
|H| \leq  \frac{1}{2} \Bigl(k + \Big\lfloor \frac{k-1}{r} \Big\rfloor\Bigr)|\partial H|.
\end{equation}
To prove (\ref{hbound2}), we reduce the case $r$ is odd to the case $r$ is even, and apply (\ref{hbound}) from the last proof.
Form the $(r + 1)$-graph $H^+$ by adding a set $X$ of vertices to $V(H)$, and let $H^+ = \{\{x\} \cup e : x \in X, e \in H\}$. It is convenient to
let $\phi(\ell) = \lceil (\ell + r)/(r + 1)\rceil$ for $\ell \geq 1$.

\medskip

It is straightforward to check that if $P = v_0 v_1 \dots v_{\ell + r - 1}$ is a tight $\ell$-path in $H^+$, then $|V(P) \cap X| \leq \phi(\ell)$. In addition, the sequence of vertices $v_i \in V(P) \backslash X$
in increasing order of subscripts forms a tight path in $H$ of length at least $\ell + 1 - \phi(\ell)$. Setting $\ell = k + \lfloor (k - 1)/r \rfloor + 1$, we
have $\ell + 1 - \phi(\ell) = k$, and therefore $H^+$ has no tight $\ell$-path. By (\ref{hbound}) applied to $H^+$,
\[ |H^+| \leq \frac{\ell - 1}{2} |\partial H^+|.\]
Since $|H^+| = |X| |H|$ and $|\partial H^+| = |X||\partial H| + |H|$, we find
\[ |X||H|\leq \frac{\ell - 1}{2} |X||\partial H| + \frac{\ell - 1}{2}|H|.\]
Choosing $|X| > (\ell - 1)|H|/2$ and dividing by $|X|$, we obtain $|H| \leq (\ell - 1)|\partial H|/2$. Since $(\ell - 1)/2 = (k + \lfloor (k - 1)/r \rfloor)/2$,
this proves (\ref{hbound2}). \qed

\section{Concluding remarks}

$\bullet$ It turns out using Steiner systems with
arbitrarily large block sizes~\cite{GKLO,Keevash}) that for each fixed $k,r \geq 2$, both of the following limits exist:
\[ z(k,r) := \lim_{n \rightarrow \infty} \frac{\ex_{\circlearrowright}(n,\mathsf P_k^r) }{{n \choose r - 1}} \quad \quad \mbox{ and } \quad \quad
p(k,r) := \lim_{n \rightarrow \infty} \frac{\ex(n,P_k^r) }{{n \choose r - 1}}.\]
The first limit is determined by Theorem~\ref{mainzigzag} and the construction
in Section \ref{lb} for $k \equiv 1 \Mod{r}$, and for $r \geq 4$ the problem is wide open in all remaining cases, even for $k = 2$.

\medskip

$\bullet$ For $k \leq r + 1$, an improvement over Theorem \ref{tight-path} is possible, slightly improving the results of Patk\'{o}s~\cite{Patkos}: we prove by induction on $r$ that
if $r \geq k - 1$, the
\[ \ex(n,P_k^r) \leq \frac{k^2}{2r} {n \choose r - 1}.\]
If $r = k - 1$, this follows from Theorem \ref{tight-path}. Suppose $r \geq k$ and we have proved the bound for $(r - 1)$-graphs. Let $H$
be an $r$-graph with no tight $k$-path and pick a vertex $v \in V(H)$ contained in at least $r|H|/n$ edges of $H$.
Consider the link hypergraph $H_v = \{e \in \partial H : e \cup \{v\} \in H\}$. Then
$H_v$ has no tight $k$-path, otherwise adding $v$ to each edge we get a tight $k$-path in $H$. By induction,
\[ \frac{r|H|}{n} \leq |H_v| \leq \frac{k^2}{2(r - 1)} {n - 1 \choose r - 2} \leq \frac{k^2}{2n} {n \choose r - 1} \]
and this implies $|H| \leq \frac{k^2}{2r} {n \choose r - 1}$, as required.

\medskip

$\bullet$ It is possible when $r \geq 3$ is odd to obtain a very slight improvement over Theorem \ref{tight-path}, namely
\[ \ex(n,P_k^r) \leq \frac{1}{r}(\sqrt{a} + \sqrt{b})^2 {n \choose r - 1}\]
where $a = \lfloor (k - 1)/r \rfloor$ and $b = (r - 1)(k - 1 - a)/2$ and $n$ is sufficiently large. For the purpose of comparison,
we obtain
\[
p(k,r) \leq k \cdot \Bigl(\frac{1}{2} + \frac{\sqrt{2} - 1}{r} + c\Bigr)  \]
where $c = O(r^{-2})$. For $r = 3$, we find that the upper bound is at most $\frac{1}{9}(3 + \sqrt{8})k\cdot {n \choose 2}$.

\medskip

$\bullet$ The proof in Section 5 shows that if $s = r/2$, $n$ is a multiple of $s$, and $G$ is an $n$-vertex $r$-graph such that $V(G)$ is partitioned into sets $B_0,B_1,\dots,B_{s-1}$ with $|B_i| = n/s$ and
$|e \cap B_i| = 2$ for $0 \leq i < s$ and every edge $e \in G$, then $|G| \leq 2^{s - 1} (k - 1)(n/r)^{r-1}$, and this is asymptotically tight if $k \equiv 1 \Mod{r}$. Indeed, let $B_0,B_1,\dots,B_{s-1}$ be disjoint sets of
size $n/s$, and let $A_i \subset B_i$ have size $(k - 1)/r$. Then let $G_i$ consist of all $r$-sets with one vertex in $A_i$, one vertex in $B_i \backslash A_i$,
and two vertices in each $B_j \backslash A_j$ for $0 \leq j < s, j \neq i$. Let $G = \bigcup_{i = 0}^{s-1} G_i$. Then $|e \cap f| \leq r - 2$ for $e \in G_i$ and $f \in G_j$ with $i \neq j$, so if $G$
contains a tight $k$-path, then the tight $k$-path is contained in some $G_i$. However, $A_i$ is a transversal of each $G_i$, so $G_i$ cannot contain a tight $k$-path. Therefore $G$ has no tight $k$-path, and furthermore
\[ |G| = \sum_{i=0}^{s-1}|G_i| = s\frac{k-1}{r}\frac{n}{s} {n/s \choose 2}^{s-1} + O(n^{r-2}) = 2^{s - 1}(k - 1) \Bigl(\frac{n}{r}\Bigr)^{r - 1} + O(n^{r - 2}).\]

\medskip

$\bullet$ In forthcoming work, we consider extremal problems for various other analogs of paths and matchings in the setting of convex geometric hypergraphs,
having considered only zigzag paths and stacks of even uniformity in this paper.

%%% AUTHOR: optional appendix here
%\appendix %% you may comment this out if no Appendix
%\section*{Appendix}
%\section{Improving the constants}
%Material is placed here as needed.

%%% AUTHOR: optional acknowledgments here
\section*{Acknowledgments} %%  you may comment this out if no Ackno
The last author would like to thank Gil Kalai for stimulating discussions occuring at the Oberwolfach Combinatorics Conference, 2017.
This research was partly conducted during AIM SQuaRes (Structured Quartet Research Ensembles) workshops, and we gratefully acknowledge the support of AIM. The authors are also grateful to a referee for  carefully reading  the paper and providing many valuable comments that improved the presentation.

%%% AUTHOR:
%%% Bibliography goes here. Note that the arXiv cannot process bibtex
%%% or biber bibliographies.  Example of acceptable bibliograpy format:
\bibliographystyle{amsplain}

%% AUTHOR: You can generate such a bibliography from a .bib file by
%% running pdflatex/bibtex/pdflatex/pdflatex and then pasting the .bbl file
%% between \begin{thebibliography} and \end{bibliography}

%%% AUTHOR: Include a short description of each author following the
%%% structure below. Use the same short tags used previously.
%%% Use \imageat{} and \imagedot{} instead of "@" and "." in
%%% email addresses-this replaces the symbols with graphics to avoid
%%% e-mail address harvesting from the .pdf file

\begin{aicauthors}
\begin{authorinfo}[furedi]
		Zolt\'an F\" uredi \\
	Alfr\' ed R\' enyi Institute of Mathematics \\
	Hungarian Academy of Sciences \\
	Re\'{a}ltanoda utca 13-15. \\
	H-1053, Budapest, Hungary. \\
	zfuredi\imageat{}gmail\imagedot{}com
\end{authorinfo}

\begin{authorinfo}[jiang]
	 Tao Jiang \\
Department of Mathematics \\
 Miami University \\
 Oxford, OH 45056, USA. \\
 jiangt\imageat{}miamioh\imagedot{}edu
\end{authorinfo}
	
\begin{authorinfo}[kostochka]
 Alexandr Kostochka \\
University of Illinois at Urbana--Champaign \\
Urbana, IL 61801 \\
and Sobolev Institute of Mathematics \\
Novosibirsk 630090, Russia. \\
kostochk\imageat{}math\imagedot{}uiuc\imagedot{}edu
\end{authorinfo}

\begin{authorinfo}[mubayi]
 Dhruv Mubayi \\
Department of Mathematics, Statistics \\
and Computer Science \\
University of Illinois at Chicago \\
Chicago, IL 60607. \\
 mubayi\imageat{}uic\imagedot{}edu
\end{authorinfo}

\begin{authorinfo}[verstraete]
	Jacques Verstra\"ete \\
Department of Mathematics \\
University of California at San Diego \\
9500 Gilman Drive, La Jolla, California 92093-0112, USA. \\
jverstra\imageat{}math\imagedot{}ucsd\imagedot{}edu.
\end{authorinfo}

\end{aicauthors}

\end{document}